\documentclass[12pt]{amsart}
 \usepackage{adjustbox}   
\usepackage{amsmath, amssymb,mathtools,amsthm,amscd,
multirow, tikz-cd,blindtext, graphicx, float, tabularx}
\usepackage[all]{xy}
\usepackage[english]{babel}
\usepackage[colorinlistoftodos]{todonotes}
\definecolor{dodgerblue}{rgb}{0.12, 0.56, 1.0}
\usepackage[raiselinks=false,colorlinks=true,citecolor=blue,urlcolor=dodger blue,
linkcolor=blue,bookmarksopen=true,pdftex]{hyperref}

\tikzcdset{scale cd/.style={every label/.append style={scale=#1},
    cells={nodes={scale=#1}}}}
\usepackage[mathscr]{eucal}
\newtheorem{theorem}{Theorem}[section]
\newtheorem{proposition}[theorem]{Proposition}
\newtheorem{lemma}[theorem]{Lemma}

\begin{document}

\newcommand{\ind}{{\rm ind}}
\newcommand{\lra}[1]{ \langle #1 \rangle}
\def\inn{\mathrm{Inn}}
\def\myhom{\mathrm{Hom}}
\def\gcd{\mathrm{gcd}}
\newcommand{\tmod}[1]{\; (#1)}
\newcommand{\myendo}[2]{{\rm End}(#1,#2)}
\newcommand\citeY[1]{\citeauthor{#1} (\citeyear{#1})}
\newcommand\citeN[1]{[\citenumber{#1}]}

\newcommand{\af}{\alpha}
\newcommand{\et}{\eta}
\newcommand{\ga}{\gamma}
\newcommand{\Ga}{\Gamma}
\newcommand{\ta}{\tau}
\newcommand{\ph}{\varphi}
\newcommand{\bt}{\beta}
\newcommand{\lb}{\lambda}
\newcommand{\Lb}{\Lambda}
\newcommand{\wh}{\widehat}
\newcommand{\sg}{\sigma}
\newcommand{\Sg}{\Sigma}
\newcommand{\om}{\omega}
\newcommand{\Om}{\Omega}

\newcommand{\bb}{\mathbb}
\newcommand{\bbN}{\mathbb N}
\newcommand{\bbZ}{\mathbb{Z}}
\newcommand{\bbR}{\mathbb R}
\newcommand{\bbC}{\mathbb C}

\newcommand{\cH}{\mathcal H}
\newcommand{\cF}{\mathcal F}
\newcommand{\cN}{\mathcal N}
\newcommand{\cR}{\mathcal R}
\newcommand{\cc}{\mathcal C}

\newcommand{\T}{\mathrm{T}}
\newcommand{\im}{\mathrm{Im}}
\newcommand{\SO}{\mathrm{SO}}
\renewcommand{\O}{\mathrm{O}}
\newcommand{\Spin}{\mathrm{Spin}}
\newcommand{\E}{\mathrm E}
\newcommand{\tor}{\mathrm{Tor}}
\newcommand{\rank}{\mathrm{rank}}
\newcommand{\Mod}[1]{\ (\mathrm{mod}\ #1)}
\newcommand{\GL}{\mathrm{GL}}
\newcommand{\SU}{\mathrm{SU}}
\newcommand{\dprime}{{\prime\prime}}
\newcommand{\p}{\prime}

\newcommand{\bea} {\begin{eqnarray*}}
\newcommand{\beq} {\begin{equation}}
\newcommand{\bey} {\begin{eqnarray}}
\newcommand{\eea} {\end{eqnarray*}}
\newcommand{\eeq} {\end{equation}}
\newcommand{\eey} {\end{eqnarray}}
\newcommand{\ovl}{\overline}
\newcommand{\I}{{\imath}}

\newcolumntype{P}[1]{>{\centering\arraybackslash}p{#1}}

\title{The K-ring of $\E_6/\Spin(10)$}
\author[S. Podder]{Sudeep Podder}
\address{Department of Mathematics, Indian Institute of Technology Madras, Chennai 600036, India}
\email{sudeep@smail.iitm.ac.in}
\author[P. Sankaran]{Parameswaran Sankaran}
\address{Chennai Mathematical Institute, SIPCOT IT Park, Siruseri, Kelambakkam, 603103, India}
\email{sankaran@cmi.ac.in}

\date{\today}

\begin{abstract}
Let $\E_6$ denote the simply-connected compact exceptional Lie group of rank $6$.  The Lie group $\Spin(10)$ naturally embeds in $\E_6$, corresponding to the inclusion of the Dynkin diagrams.  We determine the $K$-ring of the coset space $\E_6/\Spin(10)$. We identify the class of the tangent bundle of $\E_6/\Spin(10)$ in $KO(\E_6/\Spin(10))$.  As an application we show that $\E_6/\Spin(10)$ 
can be immersed in the Euclidean space $\mathbb R^{53}.$  
\end{abstract}
\maketitle

\section{Introduction}
The purpose of this article is to determine the ring structure of the complex 
$K$-theory of the compact homogeneous space $M=\E_6/\Spin(10)$ where $\Spin(10)$ 
is embedded in the compact simply-connected Lie group $\E_6$ obtained from the embedding of the Dynkin diagram of $\Spin(10)$ 
into that of $\E_6$.  There are two embeddings of the Dynkin diagram of $\Spin(10)$  corresponding to omitting any of the two end roots of the `long arms' of the Dynkin diagram of $\E_6$. 
 They are mapped to one another by a diagram automorphism of $\E_6$, so the corresponding embeddings of $\Spin(10)$ are mapped to one another 
by an outer automorphism of $\E_6$.  So 
the corresponding homogeneous spaces are diffeomorphic.  The manifold $M$ is the principal $\mathbb S^1$ bundle over $X=\E_6/(\Spin(10)\cdot\mathbb S^1)$,  which is a Hermitian symmetric space.  It is known that 
$M\times \mathbb S^1$ admits the structure of the total space of a complex analytic bundle with fibre an elliptic curve.  See \cite{wang}.  
However, $M\times \mathbb S^1$ does not admit a Kähler structure 
since its first Betti number is $1$.  The $K$-ring of $X$ is known.  Indeed if $Y$ is any homogeneous space $G/H$ where $G$ is a compact simply connected Lie group and $H$, a closed connected 
Lie subgroup of $G$ having the same rank as $G$, then it 
is known that $K^0(Y)=RH\otimes_{RG} \mathbb Z$.   Here $RG$ is the complex representation ring of $G$, $RH$ is an $RG$-module via 
the restriction homomorphism $RG\to RH$,  and $\mathbb Z$ is the $RG$-module via the augmentation $RG\to \mathbb Z$ that maps $[V]$ to $\dim V$ for any complex representation $V$ of $G$.   See \cite{AH1} and \cite{P}.

We now state the main results of this note.
Denote by $\lambda_1$ the standard representation of $\SO(10)$ on $\mathbb C^{10}$, 
viewed as a representation of $\Spin(10)$ via the double covering $\Spin(10)
\to \SO(10)$.  We shall denote by the same symbol $\lambda_1$ the $\alpha$-construction on $\lambda_1$.  (Thus the total space $E(\lambda_1)$ equals 
$\E_6\times_{\Spin(10)}\mathbb C^{10}$ with projection $E(\lambda_1)\to \E_6/\Spin(10)$ defined as $[x,v]\mapsto x\Spin(10)~\forall x\in \E_6, v\in \mathbb C^{10}.$)
\begin{theorem}\label{main1}
$K^0(\E_6/\Spin(10))\cong \mathbb Z[u]/\langle u^3\rangle$ where 
$u=\lambda_1-10$. The $K^1(\E_6/\Spin(10))$ is a free $K^0(\E_6/\Spin(10))$-module of rank $1$. 
\end{theorem}

We identify the 
class of the tangent bundle of $\E_6/\Spin(10)$ in its $KO$-ring and 
obtain the following.  

\begin{theorem}\label{main2}
The manifold $\E_6/\Spin (10)$ can be immersed in $\mathbb R^{53}$ but not in 
$\mathbb R^{40}.$
\end{theorem} 
The non-immersion result is essentially due to Singhof and Wemmer \cite[\S 3]{SW}.

Our proof of Theorem \ref{main1} involves use of the Hodgkin spectral sequence \cite{HO}, which 
will be recalled briefly in \S\ref{hss}.  
As a crucial computational aid, we will 
also need the `change of ring theorem', namely \cite[Theorem 6.1, Chapter 16]{CE}. 
  The $E_2$-page of the Hodgkin spectral sequence involves 
the representation ring of $\E_6$ and the structure of 
$R\Spin(10)$ as a module over $R\E_6$ via the restriction homomorphism $R\E_6\to R\Spin(10)$.   They will be described in \S \ref{reste6Tospin}.   
 For the sake of completeness, 
we describe $\E_6$ as a subgroup of the compact Lie group $\E_8$.  This will also help us 
describe the embedding of $\Spin(10)$ in $\E_6$.  The main theorems are proved in \S\ref{changeofrings} and 
\S\ref{immersion}.  

We have chosen to study the $K$-theory of $\E_6/\Spin(10)$ since it is the smallest dimensional quotient $\E_6/H$ where $H$ is connected and of rank less than $6$. Although the tools employed here would be applicable in full generality, the computation of the $E_2$-page $\tor^*_{R\E_6}(RH, \mathbb Z)$ appears to be rather unwieldy.  

We assume familiarity with the basic notions of finite dimensional representations of compact Lie groups as can be found in \cite{A6}, \cite{FH}, the Clifford algebra and the spin representations (as in \cite{Hu}), as well as 
the description of the exceptional Lie group $\E_6$. A good reference for the construction of the exceptional simple Lie groups is the book 
\cite{AdamsMahmudMimura}.  

\section{$\E_6$ as a subgroup of $\E_8$}
We begin by describing the compact simply connected exceptional Lie group $\E_6$ as a subgroup of the compact exceptional Lie group $\E_8$.  We follow \cite{AdamsMahmudMimura} closely.
This will allow us to describe 
the `standard' maximal torus of $\E_6$ as a subgroup of the `standard' maximal torus of $\E_8$.  This will be needed to carry out our description of the restriction homomorphism $R\E_6\to R\Spin(10)$.    

The subgroup $\E_6$ of $\E_8$ is the centralizer of a copy of $\SU(3)$ 
contained in $\E_8$.  In order to describe this subgroup (to be identified with $\SU(3)$) and 
also to describe the fundamental representations of 
$\E_6$, we need the description of the Lie algebra $Lie(\E_8)$.   To fix notations, 
we need to label of nodes of the Dynkin diagram of $\E_8$.
We follow the labelling convention as in \cite{AdamsMahmudMimura}, given below.  

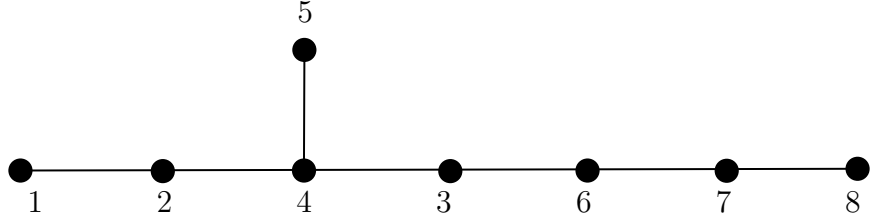
\begin{figure}[H]
\centering
\tikzset{every picture/.style={line width=0.75pt}}        

\begin{tikzpicture}[x=0.75pt,y=0.75pt,yscale=-1,xscale=1]

\draw    (184.2,121.31) -- (612,120.42) ;
\draw  [fill={rgb, 255:red, 0; green, 0; blue, 0  }  ,fill opacity=1 ] (184.2,121.31) .. controls (184.2,118.22) and (186.7,115.71) .. (189.8,115.71) .. controls (192.89,115.71) and (195.4,118.22) .. (195.4,121.31) .. controls (195.4,124.4) and (192.89,126.91) .. (189.8,126.91) .. controls (186.7,126.91) and (184.2,124.4) .. (184.2,121.31) -- cycle ;
\draw  [fill={rgb, 255:red, 0; green, 0; blue, 0  }  ,fill opacity=1 ] (470.2,121.31) .. controls (470.2,118.22) and (472.7,115.71) .. (475.8,115.71) .. controls (478.89,115.71) and (481.4,118.22) .. (481.4,121.31) .. controls (481.4,124.4) and (478.89,126.91) .. (475.8,126.91) .. controls (472.7,126.91) and (470.2,124.4) .. (470.2,121.31) -- cycle ;
\draw  [fill={rgb, 255:red, 0; green, 0; blue, 0  }  ,fill opacity=1 ] (327.2,121.31) .. controls (327.2,118.22) and (329.7,115.71) .. (332.8,115.71) .. controls (335.89,115.71) and (338.4,118.22) .. (338.4,121.31) .. controls (338.4,124.4) and (335.89,126.91) .. (332.8,126.91) .. controls (329.7,126.91) and (327.2,124.4) .. (327.2,121.31) -- cycle ;
\draw  [fill={rgb, 255:red, 0; green, 0; blue, 0  }  ,fill opacity=1 ] (256,121.6) .. controls (256,118.51) and (258.51,116) .. (261.6,116) .. controls (264.69,116) and (267.2,118.51) .. (267.2,121.6) .. controls (267.2,124.69) and (264.69,127.2) .. (261.6,127.2) .. controls (258.51,127.2) and (256,124.69) .. (256,121.6) -- cycle ;
\draw  [fill={rgb, 255:red, 0; green, 0; blue, 0  }  ,fill opacity=1 ] (401,121.6) .. controls (401,118.51) and (403.51,116) .. (406.6,116) .. controls (409.69,116) and (412.2,118.51) .. (412.2,121.6) .. controls (412.2,124.69) and (409.69,127.2) .. (406.6,127.2) .. controls (403.51,127.2) and (401,124.69) .. (401,121.6) -- cycle ;
\draw    (333,66.02) -- (332.8,115.71) ;
\draw  [fill={rgb, 255:red, 0; green, 0; blue, 0  }  ,fill opacity=1 ] (327.4,60.42) .. controls (327.4,57.33) and (329.9,54.83) .. (333,54.83) .. controls (336.09,54.83) and (338.59,57.33) .. (338.59,60.42) .. controls (338.59,63.52) and (336.09,66.02) .. (333,66.02) .. controls (329.9,66.02) and (327.4,63.52) .. (327.4,60.42) -- cycle ;
\draw  [fill={rgb, 255:red, 0; green, 0; blue, 0 }  ,fill opacity=1 ] (540.4,121.42) .. controls (540.4,118.33) and (542.9,115.83) .. (546,115.83) .. controls (549.09,115.83) and (551.59,118.33) .. (551.59,121.42) .. controls (551.59,124.52) and (549.09,127.02) .. (546,127.02) .. controls (542.9,127.02) and (540.4,124.52) .. (540.4,121.42) -- cycle ;
\draw  [fill={rgb, 255:red, 0; green, 0; blue, 0  }  ,fill opacity=1 ] (606.4,120.42) .. controls (606.4,117.33) and (608.9,114.83) .. (612,114.83) .. controls (615.09,114.83) and (617.59,117.33) .. (617.59,120.42) .. controls (617.59,123.52) and (615.09,126.02) .. (612,126.02) .. controls (608.9,126.02) and (606.4,123.52) .. (606.4,120.42) -- cycle ;

\draw (191.8,130.31) node [anchor=north west][inner sep=0.75pt]    {$1\ \ \ \ \ \ \ \ \ \ \ 2\ \ \ \ \ \ \ \ \ \ \ \ 4\ \ \ \ \ \ \ \ \ \ \ \ 3\ \ \ \ \ \ \ \ \ \ \ \ 6\ \ \ \ \ \ \ \ \ \ \ \ 7\ \ \ \ \ \ \ \ \ \ \ 8$};
\draw (328,34.4) node [anchor=north west][inner sep=0.75pt]    {$5$};

\end{tikzpicture}
\caption{Dynkin diagram of $\E_8.$} \label{fig:M1}
\end{figure}

Consider the half-spin representations 
$\Delta^{\pm}_8$ of $\Spin(16)$.  They are complexifications of the real representations $\Delta^{\pm}_{8, \bb R}$. 
It turns out that $Lie(\Spin(16))\oplus \Delta^+_{8, \bb R}$ has the structure of a compact Lie algebra, extending the Lie algebra structure on $Lie(\Spin(16))$ and the $Lie(\Spin(16))$-module structure of $\Delta_{8,\mathbb R}^+.$  This is the Lie algebra of $\E_8$ and the group $\E_8$ is the group of automorphisms of $Lie(\Spin(16))\oplus\Delta^+_{8, \bb R}$, which is viewed as the adjoint action.  See \cite[Chapter 7]{AdamsMahmudMimura}.  Evidently, we have a homomorphism 
$\Spin(16)\to \E_8$ since $\Spin(16)$ is simply connected.
Its image is the semi-spin group $\Spin(16)/\mathbb Z_2\subset \E_8$, where $\mathbb Z_2$ is generated by the central element 
$\zeta_8:=e_1\cdots e_{16}\in \Spin(16)$. Here the $e_j$ denotes the standard basis vectors of $\mathbb R^{16}$ and are viewed as the Clifford algebra units. The reader is referred to Husemoller's book \cite[Chapter 11]{Hu} for basic information on Clifford algebras, spin groups and the (half-)
spin representations.

Consider the canonical monomorphism $\SU (3)\to \SO (6)$.  
Since $\SU(3)$ is simply connected, this monomorphism lifts to a map $\psi:\SU (3)\to \Spin (6)$. This, followed by the obvious homomorphism
$\Spin(6)\to \Spin(16) \to \E_8$,  yields a homomorphism 
$\SU(3)\to \E_8$.  It is, in fact, a monomorphism and corresponds to the inclusion of 
 the Dynkin diagram of $\SU(3)$, regarded as the edge joining the last two nodes of the long arm of the 
Dynkin diagram of $\E_8$.
These are labelled $7$ and $8$ in Figure \ref{fig:M1}.  
       
\begin{figure}[H]
    \[\begin{tikzcd}[scale cd=0.8]
            &&\widetilde{\rm T}_3\arrow[hookrightarrow]{d}\arrow[hookrightarrow]{r}& \widetilde{\rm T}_5\times \widetilde{\rm T}_3\arrow[hookrightarrow]{d}\arrow[hookrightarrow]{r} &  \widetilde{\rm T}_8\arrow[hookrightarrow]{d}\arrow[r, ""]& \rm T^8\arrow[hookrightarrow]{d}\\
         &\SU(3)\arrow[hookrightarrow]{d}\arrow[r, "\psi"]&\Spin(6)\arrow[d, "\pi"]\arrow[hookrightarrow]{r}& \Spin(10)\times \Spin(6)\arrow[r, ""] &  \Spin(16)\arrow[r, ""]& \E_8\\
         & \mathrm U (3) \arrow[r, ""] &\SO(6) 
    \end{tikzcd}\]
\caption{} \label{fig:M3}
\end{figure}

 The exceptional group $\E_6$ is then the centralizer of the image of $\SU(3)$ in $\E_8$. It is a simply connected group, and its centre is a cyclic group of order $3$.  The Dynkin diagram of $\E_6$ is obtained by omitting the nodes $7$ and $8$.   We shall describe the `standard' maximal torus $\rm T^6$ of $\E_6$ and describe $R\rm T^6$ and $R\E_6$ using the description of 
 the `standard' maximal torus $\rm T^8$ of $\E_8$.  This will enable us to 
 determine the restriction homomorphism $R\E_6\to RH$ when 
 $H=\Spin(10).$

\subsection{The maximal tori of $\E_6$ and $\Spin(10)$} \label{maximaltoruse6} 
Let $\mathbb T_k=(\mathbb S^1)^k=\mathbb R^k/\mathbb Z^k$, the 
{\it standard} $k$-dimensional torus.  Let $\rm T_k$ be the subgroup $\SO(2)\times \cdots\times \SO(2)\subset \SO(2k)$. Thus $\rm T_k$ is generated by proper rotations  
on the plane $\mathbb Re_{2j-1}+\mathbb Re_{2j}\subset \mathbb R^{2k}, 1\le j\le k.$  Then $\widetilde{\rm T}_k=\pi^{-1}(\rm T_k)$ is a maximal torus of 
$\Spin(2k)$ where $\pi:\Spin(2k)\to \SO(2k)$ is the double covering projection.
One has the homomorphism $\bar\omega_k:\mathbb T_k\to \rm T_k$ defined as $(\theta_1, \ldots,\theta_k)\mapsto D(2\theta_1,\ldots, 2\theta_k):=\prod_{1\le j\le k}
D_j(2\theta_j)$ where $D_j(t)(e_{2j-1})= \cos(2\pi t)e_{2j-1} -\sin(2\pi t) e_{2j}, D_j(t)(e_{2j})=\sin(2\pi t)e_{2j-1} +\cos(2\pi t)e_{2j}$ and $D_j(t)(e_k)=e_k $ if $k\ne 2j-1,2j$.  

Then $\bar \omega_k $ lifts to a unique homomorphism $\omega_k:
\mathbb T_k\to \widetilde{\rm T}_k$ so that $\bar \omega_k=\pi\circ \omega_k$.  It is readily seen that $\ker(\bar \omega_k)=\{1,-1\}^k\subset (\mathbb S^1)^k$ and that 
$\ker(\omega_k)=\{(\varepsilon_1,\ldots,\varepsilon_k)\in \ker (\bar \omega)\mid \prod \varepsilon_j=1\}$.   The complex representation rings of $\mathbb T_k,\widetilde{\rm T}_k$, and $\rm T_k$ are related as follows: Let $u_j$ be the one dimension (complex) representation of $\mathbb T_k$ corresponding to the $j$th projection $u_j: \mathbb T_k\to \mathbb S^1=\rm U(1)$ viewed as a character.  Then 
$R\mathbb T_k=\mathbb Z[u_1^{\pm 1},\ldots, u_k^{\pm 1}]$, $R\rm T_k=\mathbb Z
[u_1^{\pm 2}, \ldots, u_k^{\pm 2}]\subset R\mathbb T_k$ and 
$R\widetilde{\rm T}_k=\mathbb Z[u_1^{\pm 2},\ldots, u_k^{\pm 2}, u_1\ldots u_k]
\subset R\mathbb T_k$.  The double covering $\widetilde{\rm T}_k\to \rm T_k$ induces 
the inclusion $R{\rm T}_k\subset R\widetilde{\rm T}_k$.  
We shall abbreviate $\omega_k$ and $\bar\omega_k$ to $\omega$ 
and $\bar \omega$ when $k$ is clear from the context.
Our main 
interest is when $k=5, 8$.

We have the diagonal embedding $j:\mathbb S^1\hookrightarrow \mathbb T_3$. Define $\delta: \mathbb S^1\to \Spin(6)$ to be 
$z\mapsto \omega_3(j(z))$.    It is readily seen that $\delta$ is a monomorphism.  
Then $S:=\delta(\mathbb S^1)$ is contained in $\widetilde{\rm T}_3$ and 
centralizes the image of $\SU(3)$ under $\psi$.  So 
we obtain a 
homomorphism $S\times \SU(3)\to \Spin(6)$ defined as $(z,A)\mapsto z\cdot \psi(A)$  

We have the natural homomorphism 
$\Spin(10)\times \Spin(6)\to \Spin(16)$ and 
composing with the homomorphism $\Spin(16) \to \E_8$ we obtain a homomorphism $\Psi$ defined as the composition
\begin{equation}
\Psi: \Spin(10)\times \mathbb S^1\times \SU(3) 
\stackrel{h}{\to} \Spin(10)\times \Spin(6)\to \Spin(16) \to \E_8
\end{equation}
where $h$ is defined as $(v,z,A)\mapsto (v ,\delta(z)\cdot\psi(A))$.  The 
homomorphism maps $\SU(3)$ monomorphically into $\E_8$ and the image  
$\Psi(\SU(3))$ corresponds to the Lie subalgebra of $\E_8$ whose Dynkin 
diagram is the edge of Dynkin diagram of $\E_8$ joining the nodes $7,8.$ 
The subgroup $\E_6\subset \E_8$ is, by definition, the centralizer of this copy of $\SU(3)$ in $\E_8$.  
Clearly, $\Psi(\Spin(10)\times \mathbb S^1)$ is contained in $\E_6$. 
It can be seen that $\Psi|_{\Spin(10)\times \mathbb S^1}$ is a local isomorphism whose kernel is the cyclic subgroup $Z$ of order $4$, generated by $(e_1\cdots e_{10},i)\in \Spin(10)\times \mathbb S^1$. 
Therefore $\Psi(\Spin(10)\times \mathbb S^1)=\Spin(10)\times_{Z}\mathbb S^1$.  The group 
$\Spin(10)$ is mapped injectively into $\E_6$ by $\Psi.$
The subgroup $\rm T^6:=\Psi(\tilde T_5\times \mathbb S^1)\subset \E_6$ 
is a maximal torus of $\E_6$ which contains the maximal torus $\widetilde{\rm T}_5$ of $\Spin(10).$  These are the {\em standard} maximal tori of $\E_6$ and $\Spin(10).$

\section{Representation Ring of $\E_6$.}\label{Simple_roots}
Recall that $\mathbb T_8=(\mathbb S^1)^8$ is the standard torus of dimension $8$, whose representation ring is the Laurent polynomial ring $R\mathbb T_8=\mathbb Z[u_1^{\pm 1},\ldots, u_8^{\pm 1}]$.  Here, $u_j$ is the class of the $1$-dimensional representation with character the $j$th coordinate projection $u_j:\mathbb T_8\to \mathbb S^1,~1\le j\le 8.$  
Recall the covering projection $\mathbb T_8\to \widetilde{\rm T}_8$, where $\widetilde{\rm T}_8$  is the `standard' maximal torus of $\Spin(16).$
The representation ring of $\widetilde{\rm T}_8$ is $\bb Z[u_1^{\pm 2}, \ldots, u_8^{\pm 2}, u_1\cdots u_8]\subset R\mathbb{ \bb T}_8$. 

The standard maximal torus $\rm T^8\subset \E_8$ is the image of $\widetilde{\rm T}_8\subset \Spin (16)$ under the 
homomorphism $\Spin(16)\to \E_8$.  Therefore $\rm T^8=\widetilde{\rm T}_8/\langle \zeta\rangle $ where 
$\zeta\in \Spin(16)$ is the central element 
$e_1\cdots e_{16}\in \Spin (16)$.   Thus $R{\rm T}^8\subset R\widetilde{\rm T}_8$ is generated by those 
characters of $\widetilde{\rm T}_8$ which map $\zeta$ to $1$.  Since $u_j^{2}(\zeta)=-1\in \mathbb S^1$ for each $j$, 
we see that $u_i^2u_j^{\pm 2}, u_i^{-2}u_j^{\pm2}\in R\rm T^8, 1\le i<j\le 8$.  Also  $u_1^{\varepsilon_1}\cdots u_8^{\varepsilon_8}$ are in $R\rm T^8$ when $\sum_{1\le j\le 8} \varepsilon_j$ is even, with $\varepsilon_j\in \{1,-1\}$.  In fact, these 240 characters generate the ring $R\rm T^8$. To see this, 
Note that these are precisely the non-zero weights of the (complexified) adjoint representation $\mathfrak e_8:=Lie(\E_8)\otimes \mathbb C$ of $\E_8$ since 
$\mathfrak e_8=(Lie(\Spin(16))\otimes\bb C)\oplus \Delta_8^+$
as a representation of $\Spin(16)$.  
Since $\E_8$ is simply connected and has trivial centre, we have $P_8=Q_8$ where $P_8, Q_8$ are 
the weight lattice and the root lattice of $\E_8$, respectively. 
As for any connected compact semisimple Lie group $G$ 
 we have $\mathbb ZQ_8\subset R\rm T^8\subset \mathbb ZP_8$ and so we must have 
equality throughout. (Here, for group $A$, $\mathbb ZA$ denotes the group ring of $A$.)
This shows that $R\rm T^8=\mathbb ZQ_8$ as asserted. 

Next, we describe $R\rm T^6$ as a quotient of $R\rm T^8$.
We have the commutative diagram where the copy of $\mathbb S^1$ 
in ${\rm T}^8\subset \E_8$ is described in \ref{maximaltoruse6}. 
\begin{figure}
            \[\begin{tikzcd}
         \widetilde{\rm T}_8\arrow[rr, ""] &&  \rm T^8 \arrow[hookrightarrow,rr ,""]&& \E_8 \\
         \widetilde{\rm T}_5\times \bb S^1\arrow[u, ""] \arrow[rr, ""] && \rm T^6\arrow[u,""]\arrow[hookrightarrow,rr ,""]&&  \E_6\arrow[u, " "]
    \end{tikzcd}\]

    \centering
    \label{fig:enter-label}
\caption{Maximal tori of $\E_6$ and $ \E_8$.} \label{fig:maxtori}
\end{figure}

    Let $\xi$ be the standard (one dimensional) representation of $\bb S^1$, corresponding to the identity map of $\mathbb S^1$. Thus $R\mathbb S^1=\mathbb Z[\xi,\xi^{-1}].$

    The discussion from \S\ref{maximaltoruse6} shows that the restriction map $R\widetilde{ \rm T}_8=\bb Z[u_1^{\pm 2}, \cdots, u_8^{\pm 2}, u_1\cdots u_8]
    \to \mathbb Z[\xi,\xi^{-1}]=R\mathbb S^1$ is given by $u_j
    \mapsto 1$ for $1\le j\le 5$ and $u_j\mapsto \xi, 6\le j\le 8$. So 
    the restriction map $R\rm T^8\to R \mathbb S^1$ is essentially given by the 
    same formula. Now the homomorphism 
    $R\rm T^8\to R(\widetilde{\rm T}_5\times \mathbb S^1)
    =R\widetilde{\rm T}_5\otimes R\mathbb S^1
    =\mathbb Z[u_1^{\pm 2},\ldots ,u_5^{\pm 2}, u_1\cdots u_5,\xi,\xi^{-1}]$  is defined by 
  \begin{equation}  \label{reste8toe6}
    u_i\mapsto u_i, 1\le i\le 5, \textrm{~and~} u_j\mapsto \xi, 6\le j\le 8.
  \end{equation}  
    In particular, $\prod_{1\le j\le 8} u_j$ restricts to $u_1\cdots u_5 \xi^3.$
  Since $\widetilde{\rm T}_5\times \mathbb S^1\to \rm T^6$ is a covering projection, $R{\rm T}^6\to R(\widetilde{\rm T}_5\times \mathbb S^1)$ is a monomorphism, which is regarded as an inclusion map. The ring $R{\rm T}^6\subset R(\widetilde {\rm T}_5\times \mathbb S^1)$ is generated by the characters 
  which are trivial on $(e_1\cdots e_{10}, i)\in \widetilde{\rm T}_5\times \mathbb S^1.$
 Since $\widetilde{\rm T}_5\times \mathbb S^1 \to \rm T^8$ factors through $\rm T^6\hookrightarrow \rm T^8$, it follows that the restriction homomorphism $R{\rm T}^8\to R\rm T^6$ is 
 given by the same formula (\ref{reste8toe6}) as above. In particular,
 $\Delta^+_8\in R\rm T^8$ restricts to 
 $\Delta_5^+\xi^{3}+\Delta^-_5\xi^{-3}+3\Delta_5^+\xi^{-1}+3\Delta_5^-\xi\in R\rm T^6$.

\subsection{Roots, weights, and fundamental representations of $\E_6$}

Since $\E_6$ is simply connected, the representation ring $R\E_6$ 
is a polynomial algebra in the classes of the fundamental 
representations of $\E_6$. 
We shall use 
a different set of polynomial generators of $R\E_6$ which are more 
convenient for describing the restriction homomorphism
$R\E_6\to R\Spin(10)$. See Proposition \ref{RepE6}.

We begin by describing the roots and the fundamental weights 
of $\E_6$. For a more detailed discussion, we refer the reader to  
\cite{AdamsMahmudMimura}.

Denote 
by $x_j/2:Lie(\bb T_8)\to \mathbb R$ the linear map 
such that $u_j(t)=\exp(\pi i x_j(t))\in \mathbb S^1$ for all $t\in Lie(\mathbb T_8)=\mathbb R^8$.  
Then, the roots of $E_8$ are 
obtained as $x_i\pm x_j, -(x_i\pm x_j), 1\le i<j\le 8$, $(\sum_{1\le j\le 8}\varepsilon_j x_j)/2$, where $\varepsilon_j\in \{1,-1\}$ with $\sum \varepsilon_j$ even.   

The root lattice of $\SU(3)\subset \E_8$ is the subgroup generated by $x_6-x_7, x_7-x_8
\in Q_8$. Since $\E_6$ is the centralizer of $\SU(3)$, 
the root lattice $Q_6$ of $\E_6$ is 
obtained as the subgroup of vectors in $Q_8\subset \mathbb R^8$ 
that are orthogonal, with respect to the standard inner product on $\mathbb R^8$, to $\{x_6-x_7, x_7-x_8\}$. Thus the set $\Phi_6$ of roots of $\E_6$
consists of $\pm (x_i\pm x_j), 1\le i<j\le 5, (\sum_{1\le j\le 8}\varepsilon_j x_j)/2$ where $\sum \varepsilon_j$ is even, $\varepsilon_j\in \{1,-1\}$ with $ \varepsilon_6=\varepsilon_7=\varepsilon_8.$  
We set $\alpha_1=\frac{1}{2}(x_1+x_2-x_3-x_4-x_5-x_6-x_7-x_8), \alpha_2=-x_2+x_3, \alpha_3=x_2+x_3, \alpha_4=-x_3+x_4, \alpha_5=-x_4+x_5, \alpha_6=\frac{1}{2}(x_1-x_2-x_3-x_4-x_5+x_6+x_7+x_8)$.  Then $\{\alpha_i\mid  1\le i\le 6\}$ is a set of simple roots where the labelling is consistent with the labelling of the nodes of the Dynkin diagram of $\E_8$. 
We shall denote the set of positive roots of $\E_6$ by $\Phi_6^+.$
The corresponding fundamental weights of $\E_6$ are 
 \begin{align*}
                \varpi_1 & =x_1-\frac{1}{3}(x_6+x_7+x_8)\\
                \varpi_2 & =\frac{3}{2}x_1+\frac{1}{2}(-x_2+x_3+x_4+x_5)-\frac{1}{6}(x_6+x_7+x_8)\\
                \varpi_3 & =\frac{3}{2}x_1+\frac{1}{2}(x_2+x_3+x_4+x_5)+\frac{1}{6}(x_6+x_7+x_8)\\
                \varpi_4 & =2x_1+x_4+x_5\\
                \varpi_5& =x_1+x_5\\
                \varpi_6 & =x_1+\frac{1}{3}(x_6+x_7+x_8).
\end{align*}

Therefore the weight lattice $P_6$ of 
$\E_6$ equals  
$\{\sum_{1\le j\le 6} a_j\varpi_j\mid a_j\in \mathbb Z\} $.  Observe that, since $\varpi_5=x_1+x_5$ is a root, $V_{\varpi_5}$ is the 
(complexified) 
adjoint representation $\mathfrak e_6.$  The dimensions of the fundamental representations of $\E_6$ 
are well-known to experts and are given in Table \ref{tab1}.  (Cf. \url{http://oeis.org/A121737}.) 
In any case, 
the dimension of an irreducible representation $V_\lambda$ with highest weight $\lambda$ 
can be obtained using the Weyl dimension formula.  See \cite[Theorem 10.18]{Hall}.

The involution of the Dynkin diagram of 
$\E_6$ fixes the nodes $4,5$ and swaps $1$ with $6$ and $2$ with $3$. The outer automorphism 
of $\E_6$, that induces the diagram automorphism, induces an involution of $R\E_6$ 
that fixes $[V_{\varpi_4}], [V_{\varpi_5}]$ and swaps $[V_{\varpi_1}]$ with $[V_{\varpi_6}]$
and $[V_{\varpi_2}]$ with $[V_{\varpi_3}]$.

\begin{table}[H]
  \centering
  \begin{tabular}{|P{7cm}|P{3.325cm}|P{3.325cm}|}
    \hline
    \textbf{Highest weight} & \textbf{Representation}                                         &   \textbf{Dimension}   \\ \hline
    $\varpi_1, \varpi_6$               & $V_{\varpi_1}, V_{\varpi_6}$              & 27                      \\ \hline
    $\varpi_2, \varpi_3$      & $V_{\varpi_2}, V_{\varpi_3}$             & 351         \\ \hline
    $\varpi_5$                                 & $V_{\varpi_5}$              & 78         \\ \hline
    $\varpi_4$                                              & $V_{\varpi_4}$              & 2925        \\ \hline
  \end{tabular}
  \newline\newline
  \caption{Dimensions of fundamental representations of $\E_6$.}\label{tab1}
\end{table}

As mentioned already, since $\E_6$ is simply connected, we have 
$R\E_6=\mathbb Z[[V_{\varpi_1}], [V_{\varpi_2}],\dots, [V_{\varpi_6}]]$, where $V_{\varpi_j}$ is 
complex representation of $\E_6$ with highest weight $\varpi_j, 1\leq j\leq 6$.
Proposition \ref{RepE6} below says that the generators $[V_{\varpi_4}]$ of $R\E_6$ can be replaced by $\Lambda^2(V_{\varpi_5})$.  
This is a particular case of a more general result of 
Adams \cite[Proposition 1]{A3}, detailed proof of which can be found in the paper of Guillot \cite[\S4]{gulliot}. 

\begin{proposition}\label{RepE6}
Set $\alpha:=[V_{\varpi_1}],\beta:=[V_{\varpi_6}],\gamma:=[V_{\varpi_5}]$.   Then 
$[V_{\varpi_2}]=\Lambda^2(\alpha)$, $[V_{\varpi_3}]=\Lambda^2(\beta)$, and 
$[V_{\varpi_4}]=\Lambda^2(\gamma)-\gamma$ and $\Lambda^3(V_{\varpi_1})\cong V_{\varpi_4}\cong \Lambda^3(V_{\varpi_6}).$
In particular, 
\[R\E_6=\mathbb Z[\alpha, \beta, \gamma,  \Lambda^2(\alpha), \Lambda^2(\beta), \Lambda^2(\ga)].\]
\end{proposition}

\begin{proof} 
It is easily seen that the highest weight that occurs in the representation $\Lambda^2(V_{\varpi_1})$ is $\varpi_1+\varpi_1-\alpha_1=2\varpi_1-\alpha_1=\varpi_2$.  It follows that 
the irreducible representation $V_{\varpi_2} $  occurs 
as a subrepresentation of $\Lambda^2(V_{\varpi_1})$.  From Table \ref{tab1}, we get that 
$\dim (V_{\varpi_2})=351=\binom{27}{2}=\dim \Lambda^2(V_{\varpi_1})$. Hence $\Lambda^2 (V_{\varpi_1})=V_{\varpi_2}$.  The same argument shows that $\Lambda^2(V_{\varpi_6})=V_{\varpi_3}$.  
As remarked already, 
$V_{\varpi_5}=\mathfrak e_6$ is the (complexified) adjoint representation 
$\E_6$.  It follows that $\Lambda^2(V_{\varpi_5})$ contains $V_{\varpi_5}$ as a 
subrepresentation.  In fact, the Lie bracket defines an alternating bilinear map $\mathfrak e_6\mathfrak \times \mathfrak e_6\to \mathfrak e_6$ namely $(X,Y)\mapsto [X,Y]$. 
Thus we obtain homomorphism $\Lambda^2(V_{\varpi_5})\to V_{\varpi_5}$ which is 
a morphism of $\E_6$-modules.   This is a surjection since $\mathfrak e_6$ is a simple Lie algebra.  By complete reducibility 
we obtain that $V_{\varpi_5}$  is a subrepresentation of 
$\Lambda^2(V_{\varpi_5})$.   Again, weight consideration 
shows that $\varpi_5+(\varpi_5-\alpha_5)=2\varpi_5-\alpha_5=\varpi_4$ is the highest weight that occurs in $\Lambda^2(V_{\varpi_5})$ and so $V_{\varpi_4} $ is a subrepresentation of $\Lambda^2(V_{\varpi_5})$.  Furthermore, $\dim(\Lambda^2(V_{\varpi_5}))={78\choose 2}=3003, \dim (V_{\varpi_5})=78$ and $\dim(V_{\varpi_4})$ is known to be $2925$.  It follows 
that
\begin{equation}\label{lambda2adjoint}
\Lambda^2(V_{\varpi_5})\cong V_{\varpi_4}\oplus V_{\varpi_5}.  
\end{equation}
We conclude that $R\E_6=\mathbb Z[\alpha,\beta,\gamma,
\Lambda^2(\alpha),\Lambda^2(\beta),
\Lambda^2(\gamma)]$. 

It remains to show that $\Lambda^3(V_{\varpi_1})\cong \Lambda^3(V_{\varpi_6})\cong V_{\varpi_4}.$
We know that $\dim V_{\varpi_4}=\dim \Lambda^3(V_{\varpi_1})=\dim \Lambda^3(V_{\varpi_6})=2925$ and that 
$3\varpi_1-2\alpha_1-\alpha_2$ (resp. $3\varpi_6-2\alpha_6-\alpha_1) $ is a (highest) weight of 
$\Lambda^3(V_{\varpi_1})$ (resp. $\Lambda^3(V_{\varpi_6}))$.  Since 
$3\varpi_1-2\alpha_1-\alpha_2=\varpi_4=3\varpi_6-2\alpha_6-\alpha_3$,  arguing as before, we have 
$V_{\varpi_4}\cong \Lambda^3(V_{\varpi_1})\cong \Lambda^3(V_{\varpi_6})$. Equation (\ref{lambda2adjoint}) can be rewritten as 
\begin{equation}\label{lambda2adjoint-2}
\Lambda^2(V_{\varpi_5})\cong \Lambda^3(V_{\varpi_6})\oplus V_{\varpi_5}.
\end{equation}
\end{proof}

We end this section with the computation of 
the restriction of the (complexified) adjoint representation $\mathfrak e_8=Lie(\E_8)\otimes \mathbb C$ of $\E_8$ to $\E_6.$  This will aid in describing certain fundamental representations of $\mathfrak e_6$ to $\Spin(10)\cdot \mathbb S^1.$  Although 
the result is known---see \cite{AdamsMahmudMimura}---we give the details for the sake of completeness. 

Since the maximal torus $\rm T^6$ 
is contained in $\Spin(10)\cdot \mathbb S^1\subset 
\E_6$, we need only describe the restriction of 
$\mathfrak e_8$ to $\Spin(10)\cdot \mathbb S^1$. 
Since $Lie(\E_8)=Lie(\Spin(16))
\oplus \Delta_{8, \bb R}^+$, this amounts to finding the 
restrictions of the $\Spin(16)$-representation $Lie(\Spin(16))\otimes \mathbb C \oplus \Delta_{8}^+$ 
to $\Spin(10)\cdot \mathbb S^1\subset \Spin(16)$.

Working in $R\mathbb T_8$, 
the characters $u_j, j=6,7,8$ restrict to 
the identity character $\xi$ of $\mathbb S^1$ and $u_i, 1\le i\le 5,$ to the trivial 
character on $\mathbb S^1$.  Each $u_j, j=6,7,8,$ restricts to $\xi$ on the $\mathbb S^1$ factor. Thus 
$u_6^{\varepsilon_6}u_7^{\varepsilon_7}u_8^{\varepsilon_8}$ restricts to $\xi^{\varepsilon_6+\varepsilon_7+\varepsilon_8}$.
Also $u_i^2u_j^{\pm 2}\mapsto u_i^2\xi^{\pm 2}$
for $1\le i\le 5<j\le 8$, and, $u_i^2u_j^{-2}\mapsto 1, u_i^2u_j^2\mapsto \xi^{\pm 4}$ for $6\le i<j\le 8$.
 It follows that $\Delta^+_8$ restricts to 
$\Delta_5^+\xi^{3}\oplus\Delta_5^-\xi^{-3}\oplus 3\Delta_5^+\xi^{-1}\oplus 3\Delta_5^-
\xi^{1}$ on $\Spin(10)\cdot\mathbb S^1$.

Among the roots of $Lie(\Spin(16))$, $(\pm x_i\pm x_j), 1\le i< j\le 5 ,$ restrict to the roots of $Lie(\Spin (10))$.  
For each fixed $j, 6\le j\le 8$, the set of roots $(\pm x_i+ x_j), 1\le i\le 5,$ correspond to the characters $u_i^{\pm 2}\xi^2$ 
which are precisely the characters of 
$\lambda_1\otimes \mathbb C_{\xi^2}$, where $\lambda_1\cong \bb C^{10}\mathbb =\mathbb R^{10}\otimes \mathbb C$ is the standard 10-dimensional representation on which $\Spin(10)$ acts  via the double covering $\Spin(10)\to \SO(10)$. Similarly 
the roots 
$\pm x_i-x_j, 1\le i\le 5<j\le 8$, correspond to the 
characters of $\mathbb C^{10}\otimes \mathbb C_{\xi^{-2}}.$
Of the remaining roots, each of the six roots 
$\pm (x_6-x_7), \pm(x_6-x_8),\pm(x_7-x_8)$ 
yields the trivial representations, and each of $\varepsilon(x_6+x_7), \varepsilon (x_6+x_8), \varepsilon(x_7+x_8)$ yields  the representation 
$\xi^{4\varepsilon}, \varepsilon\in \{1,-1\}$.  
The trivial character occurs in $Lie(\Spin(16))$
with multiplicity $8=\rank(\Spin(16))$ whereas it occurs in $Lie(\Spin(10)\cdot \mathbb S^1)$ with multiplicity $6$.  All told, we obtained that 
$Lie(\Spin(16))\otimes \mathbb C$ restricts to $Lie(\Spin(10)\cdot \mathbb S^1)\otimes \mathbb C\oplus 3\lambda_1\xi^2\oplus 3\lambda_1\xi^{-2}
\oplus 3\xi^4\oplus 3\xi^{-4}\oplus 8$. 

In the sequel, we shall denote by $\lambda_j,1\le j\le 5$ 
the representation $\Lambda^j(\lambda_1).$

From the above discussion, we conclude that, denoting  by $\rho'$ the restriction $R\E_8\to R(\Spin(10)\cdot \mathbb S^1)$,  
\begin{equation}\label{e8Toe6}
\rho(\mathfrak e_8)=
Lie(\Spin(10)\cdot \mathbb S^1)\otimes \mathbb C\oplus 
\Delta_5^+\xi^{3}\oplus\Delta_5^-\xi^{-3}\oplus \mathbb C^8\oplus  3(V\oplus V')
\end{equation} where 
$V=\lambda_1 \xi^{-2}\oplus \Delta_5^-\xi\oplus \xi^{4},
$ and $V'=\lambda_1 \xi^{2}\oplus \Delta_5^+\xi^{-1}\oplus \xi^{-4}$ are representations of $\Spin(10)\cdot\mathbb S^1.$

Using the fact that the coefficients of $x_j, 6\le j\le 8,$ in the roots of $\mathfrak e_6$ are all equal, it is easily verified that 
\begin{equation}
\mathfrak e_6\cong Lie(\Spin(10)\cdot \mathbb S^1)\otimes \mathbb C\oplus \Delta_5^+\xi^3\oplus \Delta_5^-\xi^{-3}.  
\end{equation}

It follows from (\ref{e8Toe6}) that $V\oplus V'$ are representations of $\E_6$.  

{\it Claim:} 
\begin{equation} \label{vv'}
V_{\varpi_1}\cong V=\lambda_1 \xi^{-2}\oplus \Delta_5^-\xi\oplus \xi^{4};~ V_{\varpi_6}\cong V'=\lambda_1 \xi^{2}\oplus \Delta_5^+\xi^{-1}\oplus \xi^{-4}. 
\end{equation}
It is readily seen 
that (i) 
$\dim V=27,$ and, (ii) 
$x_1-(1/3)(x_6+x_7+x_8)$ is a weight of $V$--- 
in fact 
it is the highest weight of $V.$  It turns out that 
$V$ is in fact a representation of $\E_6$ isomorphic to $V_{\varpi_1}$. 
Similarly, $V'=\lambda_1 \xi^{2}\oplus \Delta_5^+\xi^{-1}\oplus \xi^{-4}
$ is isomorphic to $V_{\varpi_6}$. 
Thus Equation (\ref{e8Toe6}) can be rewritten as 
\begin{equation}\label{e8Toe6-1}
    \tilde\rho(\mathfrak e_8)=\mathfrak e_6\oplus 3(V_{\varpi_1}
    \oplus V_{\varpi_6})\oplus \mathbb C^8
\end{equation}
where $\tilde \rho:R\E_8\to R\E_6$ is the restriction homomorphism.

\section{The Restriction Homomorphism $\rho : R\E_6\to R\Spin(10)$}
\label{reste6Tospin}
We keep the notations of the previous section.
Set $H:=\Spin(10)\cdot \mathbb S^1$.  We shall describe 
the restriction homomorphism $\rho':R\E_6\to RH$.  
Although our interest here is in the restriction 
$\rho:R\E_6\to R\Spin(10),$ the computations of $\rho'$ 
will be more transparent 
since both $\E_6$ and $H$ have the same rank. 
We have (\cite[\S12]{Hu})
$R\Spin(10)=\mathbb Z[\lambda_1, \lambda_2, \lambda_3, \Delta^+_5,\Delta^-_5]$.  We denote by $\pi_j, $
the highest weights of $\lambda_j, 1\le j\le 4.$ 
Then $\pi_1,\pi_2,$ and $\pi_3$ are fundamental weights. 
Note that $V_{\pi_1}$ is the representation of 
$\Spin(10)$ obtained via $\pi:\Spin(10)\to \SO(10)$ from the 
standard representation of $\SO(10)$ on $U:=\mathbb C^{10}$.

Although $V_{\pi_4}=\Lambda^4(U)$ is irreducible, it 
is not a fundamental representation.

$RH$ is the subring of $R(\Spin(10)\times \mathbb S^1)=\mathbb Z[\lambda_1,\lambda_2,\lambda_3,\Delta_5^\pm, \xi^\pm ]$ 
generated by the elements $[V]$ where $V$ is a representation 
of $\Spin(10)\times \mathbb S^1$ on which $(e_1\cdots e_{10}, i)$ acts trivially.  Of course, any representation of $\E_6$ when 
restricted to $H$ satisfies this requirement.

Recall from (\ref{vv'}) that $V=V_{\varpi_1}$ and $V'=V_{\varpi_6}$ restrict to the representations $V=\lambda_1\xi^{-2}+\Delta_5^-\xi+\xi^4$ and  $V^\p=\lambda_1\xi^{2}+\Delta_5^+\xi^{-1}+\xi^{-4}$ on $\Spin (10)\cdot \bb S^1$. 
Since $\mathfrak e_6\cong Lie(\Spin(10)\cdot \mathbb S^1)\otimes \mathbb C\oplus \Delta_5^+\xi^3\oplus \Delta_5^-\xi^{-3}$, as a representation of $\Spin (10)\cdot \bb S^1$,
and since $1+\lambda_2=Lie(\Spin(10)\cdot\mathbb S^1)\otimes \bb C,$ we have 

\begin{equation}\label{rhoprimeadj}
\rho'(V_{\varpi_5}) =1+\lambda_2+\Delta^+_5\xi^3+\Delta^-_5\xi^{-3}.
\end{equation}

The restriction map $\rho^\prime\colon R\E_6\to R(\Spin (10)
\cdot S^1)$ is a $\lambda$-ring homomorphism and observing that $\Lambda^2(\Delta_5^+)=\lambda_3=\Lambda^2(\Delta_5^-)$, routine calculation shows that
\begin{equation}\label{rhoprime1}
\rho^\p(\Lambda^2(V_{\varpi_1}))=\lambda_2\xi^{-4} + \lambda_3\xi^2 + \lambda_1\Delta_5^-\xi^{-1} + \lambda_1\xi^2 + \Delta_5^-\xi^5,
\end{equation}
\begin{equation}\label{rhoprime2}
\rho^\p(\Lambda^2(V_{\varpi_6}))=\lambda_2\xi^4 + \lambda_3\xi^{-2} + \lambda_1\Delta_5^+\xi + \lambda_1\xi^{-2} + \Delta_5^+\xi^{-5},
\end{equation}
\begin{equation}\label{rhoprime3}
\rho^\p(\Lambda^2(V_{\varpi_5}))=\lambda_2+(1+\lambda_2)(\Delta_5^+\xi^3+\Delta_5^-\xi^{-3})+\lambda_3\xi^6+\lambda_3\xi^{-6}+\Lambda^2(\lambda_2)
+\Delta_5^+\Delta_5^-.
\end{equation}

The following lemma allows us to express the term $\Lambda^2(\lambda_2)$ in the last equation in terms of $\lambda_1\lambda_3$ and $\lambda_4$.  
 
\vspace{0.5cm}
\begin{lemma}\label{Lambda2U2}
The relation
\begin{equation}\label{lambda2lambda2}
\lambda_1\lambda_3=\Lambda^2(\lambda_2) +\lambda_4
\end{equation}
holds in the ring $R\Spin(10)$.
\end{lemma}
\begin{proof} Let $U_j:=\Lambda^j(\mathbb C^{10}),$ the representation of $\Spin(10)$ with highest weight $\pi_j,$  
$1\le j\le 4.$

We have an obvious surjective homomorphism $f: U_1\otimes U_3\to U_4$ of $\Spin(10)$-representations defined as $u\otimes v\mapsto u\wedge v$. 
In fact the vector $w_1=e_1\otimes e_2e_3e_4-e_2\otimes e_1e_3e_4+e_3\otimes e_1e_2e_4
-e_4\otimes e_1e_2e_3$ is seen to be a highest weight vector of $U_1\otimes U_3$ 
that maps to $4e_1e_2e_3e_4.$  Here, and in what follows,
we have written $uv$ to denote $u\wedge v.$   On the other hand $w_0:=e_1\otimes e_1e_2e_3\in U_1\otimes U_3 $ is a highest weight vector whose weight equals $\pi_1+\pi_3=\pi_1+
3\pi_1-2\alpha_1-\alpha_2-\alpha_3=(2\pi_1-\alpha_1)+(2\pi_1-\alpha_1-\alpha_2)=:\omega$ which 
is the highest among all the weights of $\Lambda^2(U_2).$  Since $w_0\in 
\ker(f)$ and $f$ is a homomorphism of $\Spin(10)$-representation, the $\Spin(10)$-submodule 
generated by $w_0$, denoted $W$ is contained in $U_1\otimes U_3$ and is 
linearly disjoint with $U_4$.  So $W\oplus U_4$ is a subrepresentation of $U_1\otimes U_3$.
It turns out that $\dim W=990=\dim \Lambda^2(U_2).$  This may be verified by 
applying the Weyl dimension formula.  It follows that $\dim W+\dim U_4=990+210=1200=\dim (U_1\otimes U_3).$ So we must 
have $W\oplus U_4\cong U_1\otimes U_3.$   This completes the proof of the lemma. 
\end{proof}

Another approach to proving the above lemma which is conceptually simpler but computationally more involved, is to show the equality of the characters on both sides of 
(\ref{lambda2lambda2}).  This may be done by directly calculating the case at hand rather than 
using the Weyl character formula.  Starting with the standard basis for $U_1=\mathbb C^{10},$ has natural bases 
of $U_2,U_3, U_4$ and also the tensor products $U_1\otimes U_3$ as well as $\Lambda^2(U_2)$, consisting of 
weight vectors.  It is straightforward to 
count with multiplicity the weights of these weight 
vectors. 
 The calculations are tabulated below, but 
the details are left to the reader.  
 Indeed, this was the proof given in \cite{podder}.

\begin{table}[H]
 \centering
 \begin{tabular}{|P{3.5cm}|P{2cm}|P{2cm}|P{2cm}|P{2cm}|}
   \hline
  Weight type 
   & No. of weights  & Mult. in $\Lambda^2(U_2)$ & Mult. in $U_4$ & Mult. in $U_1\otimes U_3$\\ 
   \hline\hline 
   $\pm x_i\pm x_j\pm x_k\pm x_l;$ & $80$  & $3$ & $1$   & $4$\\ 
   $ i<j< k<l$ & & &  &\\
   \hline
   $\pm 2x_i\pm x_j\pm x_k;$ & $240$&$1$&$0$ & $1$   \\ 
    $ i< j< k.$&&&&\\
   \hline
   $\pm x_i\pm x_j; i<j.$& $40$  & $11$ & $3$  & $14$ \\ 
   \hline
    $\pm 2x_i$ & $10$ & $4$       & $0$  & $4$ \\ 
    \hline
    $0$  & $1$  & $30$         & $10$ & $40$ \\ 
    \hline
 \end{tabular}
 \newline\newline
 \caption{Weights of $\Lambda_2(U_2)$ and $U_4$.}\label{tab3}
\end{table}

Recall that the relation $\Delta_5^+\Delta_5^- = 1 + \lambda_2 + \lambda_4$ holds in $\Spin (10)$;  see \cite[Theorem 10.3, Chapter 13]{Hu}. Therefore, from the above lemma and Equation (\ref{rhoprime3}), we get that
\begin{equation}\label{rhoprimelambda2}
\rho^\p(\Lambda^2(\gamma))=1 + 2\lambda_2 + (1+\lambda_2)(\Delta_5^+\xi^3+\Delta_5^-\xi^{-3})+\lambda_3\xi^6+\lambda_3\xi^{-6}+ \lambda_1\lambda_3.
\end{equation}

Let $i^*:R(\Spin(10)\cdot \mathbb S^1)\to R\Spin(10)$ denote the homomorphism of rings induced by the inclusion 
$i\colon \Spin(10)\to \Spin(10)\cdot\bb S^1$.  Then $\rho=i^*\circ \rho^\p\colon R\E_6\to R\Spin(10)$ is the restriction homomorphism induced by the inclusion $\Spin(10)\to \E_6$.  
Then $\rho$ 
is obtained from Equations (\ref{vv'}), (\ref{rhoprimeadj}),
(\ref{rhoprime1}), (\ref{rhoprime2}), and (\ref{rhoprimelambda2}), 
using the fact that $i^*(\xi)=1$, the trivial 1-dimensional representation of $R\Spin(10)$. Therefore, we obtain the following.

\begin{proposition}\label{restrictionformula}
The restriction homomorphism $\rho : R\E_6\to R\Spin(10)$ is given as follows.
(i) ~$\alpha^\p\colon=\rho(\alpha)=1+\lambda_1+\Delta_5^+, $\\ (ii)  ~$\beta^\p\colon=\rho(\beta)=1+\lambda_1+\Delta_5^-, $\\
(iii) ~$\ga^\p\colon=\rho(\ga)=1+\lambda_2+\Delta_5^+ + \Delta_5^-,$ \\ 
(iv)~ $\rho(\Lambda^2(\alpha))=\lambda_2 + \lambda_3 + \lambda_1\Delta_5^- + \lambda_1 + \Delta_5^-,$\\
(v) ~$\rho(\Lambda^2(\beta))=\lambda_2 + \lambda_3 + \lambda_1\Delta_5^+ + \lambda_1 + \Delta_5^+,$\\ 
(vi) ~$\rho(\Lambda^2(\ga))=1+2\lambda_2+(1+\lambda_2)(\Delta_5^++\Delta_5^-)+2\lambda_3+\lambda_1\lambda_3$.\hfill$\blacksquare$
\end{proposition}

\section{Proofs of Theorems \ref{main1} and \ref{main2}}

We begin with a brief description of the Hodgkin spectral sequence and the change of rings spectral sequence.  

\subsection{Hodgkin Spectral sequence}\label{hss}
We shall now describe briefly the Hodgkin spectral sequence.  For a more detailed 
exposition we refer the reader to \cite{RO}, \cite{AGUZ}, 
\cite{BF}, or \cite{podder-sam=nkaran}. 

Let $G$ be a compact connected Lie group and $H$ be a closed subgroup.  Let $\rho: RG\to RH$ be the restriction homomorphism and let $RG\to \mathbb Z$ be the augmentation.
We regard $RH$ and $\mathbb Z$ as $RG$-modules via these homomorphisms.  The following theorem is due to Hodgkin \cite{HO}.

\begin{theorem}  Let $G$ be a compact Lie group and $H$ an closed subgroup of $G$.   Suppose that $\pi_1(G)$ is torsion-free.  There exists a spectral sequence that converges to $K^*(G/H)$ whose $E_2$-page is  
\[E_2^{p,q}=\begin{cases}
    \tor^p_{RG}(RH, \mathbb Z), &\textrm{
    ~if~} q\equiv 0\pmod 2\\
    0, &\textrm{~otherwise.}
\end{cases}\]
The differential $d_r$ has bidegree $(r,1-r)$. In particular, 
    $d_r=0$ for all $r$ even. 
\end{theorem}

In the above theorem, $\tor^p_R(M,N)$ is, by definition, $\tor_{-p}^R(M,N)$.  
In order to compute $\tor^*_{RG}(RH,\mathbb Z)$, one uses the `change of rings'
theorem.  This is a spectral sequence, referred to as the {\it change of ring spectral sequence} in the sequel.  Its $E_2$-page is defined as follows:
$E_2^{p,q}=\tor_\Omega^p(\tor_\Lambda^q(RH,\mathbb Z),\mathbb Z)$ where $\Lambda$ 
is a conveniently chosen subring of $RG$ and $\Omega$ is the quotient of 
$RG$ by the ideal generated by the kernel of the augmentation 
$\Lambda\to \mathbb Z$.  The spectral sequence converges to $\tor^*_{RG}(RH,\mathbb Z)$.  
This is a special case of a very general result established in \cite[p. 379]{CE}.   

We shall find a subalgebra $\Lambda$
of $R\E_6$ over which $R\Spin(10)$ is free.  This will allow us to 
apply the change of ring theorem to simplify the computation of  $\mathrm{\tor}^*_{R\E_6}(R\Spin(10),\mathbb Z).$  The appropriate 
choice of $\Lambda$ is given in the following lemma.

\begin{lemma}\label{freeness_E6}
Let $\Lambda$ be the subring $\mathbb Z[\alpha, \beta, \gamma, \Lambda^2(\alpha)]$ of $R\E_6$. Then $R\Spin(10)$ is a free $\Lambda$-module 
(via the restriction homomorphism) with basis $\{\lambda_1^i:i\ge 0\}$ over $\Lambda$.
\end{lemma}
\begin{proof}
Let $\Lambda^\p\colon=\rho(\Lambda) = \mathbb Z[\alpha^\p, \beta^\p, \ga^\p, \Lambda^2(\alpha^\p)]$. Let us write $\ga^\dprime=\gamma^\p-\alpha^\p-\beta^\p$. Then $\Lambda^\p=\mathbb Z[\alpha^\p, \beta^\p, \ga^\dprime, \Lambda^2(\alpha^\p)]$ and, it is enough to check that $\alpha^\p, \beta^\p, \ga^\dprime, \Lambda^2(\alpha^\p)$ are algebraically independent.

Using Proposition \ref{restrictionformula}, we have 
$\alpha^\p=f(\lambda_1)+\Delta_5^-, 
~\beta^\p=f(\lambda_1)+\Delta_5^+,~\ga^\dprime= g(\lambda_1) +\lambda_2$ and $\Lambda^2(\alpha^\p)=h(\lambda_1, \lambda_2, \Delta_5^-)+ \lambda_3$, where $f, g, h$ are polynomials in indicated variables. 
Suppose that $F(\alpha^\p, \beta^\p, \ga^\dprime, \Lambda^2(\alpha^\p))=0$.   Since $\lambda_3, \lambda_1, \lambda_2, \Delta_5^-, \Delta_5^+$ are algebraically independent, and since $\Lambda^2(\alpha')$ is involves 
$\lambda_3$ and the others, $\alpha',\beta',\gamma''$ do not involve $\lambda_3$, 
we see that $F$ must be 
independent of $\Lambda^2(\alpha')$. 
Since there is no polynomial relation among $\lambda_2, \Delta_5^+, \Delta_5^-$ over $\mathbb Z[\lambda_1]$, we conclude that $F$ must be the zero polynomial. This proves the algebraic independence of $\alpha^\p, \beta^\p, \ga^\p, \Lambda^2(\alpha^\p)$.

We have, $\Delta_5^-=\alpha^\p-f(\lambda_1), \Delta_5^+=\beta^\p-f(\lambda_1), \lambda_2=\ga^\p-\alpha^\p-\beta^\p-g(\lambda_1)$. By substituting $\Delta_5^-$ and $\lambda_2$ from the above in $\lambda_3= \Lambda^2(\alpha^\p)-\lambda_1-(1+\lambda_1)\Delta_5^--\lambda_2$ we see $R\Spin(10)\cong\Lambda^\p[\lambda_1]$, as a $\Lambda^\p$-module.
\end{proof}

\subsection{Computation of $\tor^*_{R\E_6}(RH,\mathbb Z)$} \label{changeofrings}
In order to compute the $E_2$-page of 
the Hodgkin spectral sequence 
$ \tor^*_{R\E_6}(RH,\mathbb Z)$, which  
converges to $K^*(\E_6/\Spin(10))$, 
we will apply the change of ring spectral sequence  
(\cite[\S5,Chapter XVI]{CE}), in the case following situation (in the notation 
of \cite{CE}):
\[\Gamma=R\E_6,  A=R\Spin (10),~ K=C=\mathbb Z \textrm{~and~} \Lambda:=\mathbb Z[\alpha, \beta, \ga, \Lambda^2(\alpha)]\subset \Gamma=R\E_6.\]
Then $A$ is a free $\Lambda$-module via the restriction homomorphism, in view of Lemma \ref{freeness_E6}.    
Then 
\begin{equation}\label{changering}
\textrm{Tor}^\Omega_*(\textrm{Tor}_*^\Lambda (R\Spin(10),\mathbb Z),\mathbb Z) \implies \textrm {Tor}^{R\E_6}_*(R\Spin(10),\mathbb Z),
\end{equation}
where 
$\Omega=R\E_6/\langle \alpha-27,\beta-27,\gamma-78,\Lambda^2(\alpha)-351\rangle$.  
In view of Lemma \ref{freeness_E6},
we have $\tor^\Lambda_*(R\Spin(10),\mathbb Z)=\tor^\Lambda_0(R\Spin(10),\mathbb Z)=R\Spin(10)\otimes_\Lambda \mathbb Z$ and 
the above spectral sequence collapses. 
So we have $\tor^{R\E_6}_*(R\Spin(10),\mathbb Z)\cong \tor^\Omega_*(B,\mathbb Z)$
where 
\begin{equation}\label{B}
B_q=\textrm{Tor}^\Lambda_q(R\Spin(10),\mathbb Z)=\begin{cases} RH\otimes _\Lambda \bb Z,
&\textrm{if~}q=0,\\
0,\textrm{~if~} q\ne 0.
\end{cases}.
\end{equation}

Thus, 
\[B=B_0=RH/\langle \alpha^\p-27, \beta^\p-27, \ga^\p-78, \Lambda^2(\alpha^\p)-{27\choose 2} \rangle.\] 

In view of Lemma \ref{freeness_E6}, $B$ is isomorphic to the 
polynomial algebra 
$\mathbb Z [\lambda_1].$

We shall continue to denote by $\lambda_j, j=1,2,3$, the images in $B$ of the elements 
$\lambda_j\in RH.$  

\begin{lemma}\label{Res}
The following relations hold in $B$:\\
(a) $\Delta_5^+ = 26-\lambda_1=\Delta_5^-$,\\
(b) $\lambda_2 = 25+2\lambda_1$,\\
(c) $\lambda_3 = \lambda_1^2-28\lambda_1+300$,\\
(d) $\Lambda^2(\beta^\p)-{27\choose 2}=0$,\\
(e) $\Lambda^2(\ga^\p)-{78\choose 2}=(\lambda_1-10)^3$.
\end{lemma}
\begin{proof}
The proof of (a) and (b) are immediate from 
Proposition \ref{restrictionformula}(i) and (ii), respectively.   

In view of Proposition \ref{restrictionformula}
(iv), we have the relation ${27\choose 2}=\Lambda_2(\alpha')=\lambda_2+\lambda_3+\lambda_1+(\lambda_1+1)\Delta_5^-$ in $B$.  Using (a) and (b) we obtain 
${27\choose 2}-\lambda_3= (25+2\lambda_1)+(\lambda_1+1)(26-\lambda_1) +\lambda_1.$   
Now (c) follows up on simplification.   

(d). Using Proposition \ref{restrictionformula}(i) and (ii) again, we have $\Delta_5^+-\Delta_5^-=\beta^\p-\alpha^\p=0$ in $B$.  Now Proposition \ref{restrictionformula} (iv) and (v) shows that 
$\Lambda^2(\beta')=\Lambda^2(\alpha').$ 
 Therefore, $\Lambda^2(\beta^\p)-{27\choose 2}=\Lambda^2(\alpha^\p)-{27\choose 2}=0$ in $B$. 

(e) Recall from Proposition \ref{restrictionformula}(iii), (vi) that  the relations 
$\gamma'=1+\lambda_2+\Delta_5^++\Delta_5^-$ and 
$\Lambda^2(\ga^\p)=1+2\lambda_2+(1+\lambda_2)(\Delta_5^++\Delta_5^-)+2\lambda_3+\lambda_1\lambda_3$ hold in $\Spin(10).$  Using relations $\gamma'=78$ and the relations  (a) to (d) of the lemma which have just been proved,  
we obtain 
\begin{align*}
\Lambda^2(\ga^\p) & = (1+\lambda_2+\Delta_5^++\Delta_5^-)+\lambda_2(1+\Delta_5^++\Delta_5^-)+2\lambda_3+\lambda_1\lambda_3\\
                      & = \ga^\p+\lambda_2(\ga^\p-\lambda_2)+2\lambda_3+\lambda_1\lambda_3\\
                      & = 78 + \lambda_2(78-\lambda_2) + 2\lambda_3 + \lambda_1\lambda_3\\
                      & = 78 + (2\lambda_1+25)(53-2\lambda_1)+ (\lambda_1+2)(\lambda_1^2-28\lambda_1+300)\\
                      & = \lambda_1^3-30\lambda_1^2 +300\lambda_1+2003\\
                      & = (\lambda_1-10)^3+3003.\\ 
\end{align*}
Since $\gamma'=78$ we have $\Lambda^2(\gamma')={78\choose 2}=3003$ in $B$.  The last equation now implies that $(\lambda_1-10)^3=0$.
\end{proof}

We are now ready to prove Theorem \ref{main1}.

\noindent{\it Proof of Theorem \ref{main1}}:
Since $B_q=0$ for all $q>0$, the change of rings spectral sequence $\textrm{Tor}_*^\Omega(\textrm{Tor}^\Lambda(RH,
\mathbb Z),\mathbb Z)=\textrm{Tor}^\Omega_*(B_*,\mathbb Z)$ which converges to $\textrm{Tor}^{R\E_6}_*(R\Spin(10),\mathbb Z)$ collapses 
and we get $\tor^{R\E_6}_q(R\Spin(10),\mathbb Z)\cong \tor^\Omega_q(B,\mathbb Z)$, where 
$\Omega=R\E_6/\langle \alpha-27,\beta-27,\gamma-78,\Lambda^2(\alpha)-351\rangle=\mathbb Z[\Lambda^2(\beta), \Lambda^2(\gamma)]$.

 Let us write $x=\Lambda^2(\beta)-{27\choose 2}$ and $y=\Lambda^2(\ga)-{78\choose 2}$. Then $\Omega=\mathbb Z[x, y]$, and the Koszul resolution $\mathcal C$ of the polynomial 
 ring $\Omega$ is 
\[0\to\Omega\xrightarrow{d}\Omega\oplus\Omega\xrightarrow{d}\Omega\xrightarrow{\epsilon} \mathbb Z\to 0\]
where $\epsilon:\Omega\to \mathbb Z$ is the augmentation map defined as $x\mapsto 0, y\mapsto 0$. We shall denote by $X,Y$ the basis of $\Omega_1:=\Omega\oplus \Omega$ so that 
the basis of $\Omega_2\cong \Omega$ is $X\wedge Y.$  The differentials of the Koszul complex are defined by 
$d(X)=x=\Lambda^2(\beta^\p)-{27\choose 2}, ~d(Y)=y=\Lambda^2(\gamma^\p)-{78\choose 2}$ and $d(X\wedge Y)=d(X)Y-d(Y)X.$
Therefore, since $B$ is a $\Omega$-module via the homomorphism $\bar{{\rho}}$ induced by $\rho$, on tensoring with $B$, we obtain the following chain complex $\mathcal{C}\otimes B$: \[0\to B\xrightarrow{\bar{d}} B\oplus B\xrightarrow{\bar{d}} B\to 0\]
where, in view of Lemma \ref{Res},  $\bar{d}(X)=\bar{\rho}(x)=0, ~\bar{d}(Y)=\bar{\rho}(y)=(\lambda_1-10)^3$ and $\bar{d}(X\wedge Y)=\bar{\rho}(x)Y-X\bar{\rho}(y)=(\lambda_1-10)^3X$.  Therefore, $\tor_0^{\Omega}(B, \mathbb Z)=H^0(\mathcal{C}\otimes B)=B/\langle(\lambda_1-10)^3\rangle$ and $\tor_1^{\Omega}(B, \mathbb Z)=H^1(\mathcal{C}\otimes B)=BX/\langle(\lambda_1-10)^3X\rangle\cong B/\langle(\lambda_1-10)^3\rangle$. 
Let 
\begin{equation}\label{generatorK}
    u=\lambda_1-10.
\end{equation} 
Since $\tor^*_\Omega(B, \bb Z)=\tor_{-*}^\Omega(B,\mathbb Z)$ is generated, as a ring, by elements of 
degree $\ge -1,$  
by \cite[Proposition, \S1.3]{RO}, we conclude that $\tor_*^\Omega(B;\mathbb Z)$
is isomorphic to $K^*(\E_6/\Spin(10)).$
That is $K^0(\E_6/\Spin(10))=\mathbb Z[u]/\langle u^3\rangle,$ and 
$K^{1}(\E_6/\Spin(10))=K^0(\E_6/\Spin(10))\cdot X$,
the free $K^0(\E_6/\Spin(10))$ of rank $1$ generated by $X.$
This completes the proof of 
Theorem \ref{main1}. \hfill $\blacksquare$

\subsection{Complexification and realification}\label{candr}
Let $c:RO(G)\to RG$ and $r:RG\to RO(G)$ denote the 
 complexification and the `forgetful' homomorphism 
 obtained by restricting the scalar multiplication from $\mathbb C$ to $\mathbb R$ respectively. The homomorphism $r$ will be referred to as the realification homomorphism. 
 We shall use the same symbols for the analogous homomorphisms 
 $KO^0(X)\to K^0(X)$ and $K^0(X)\to KO^0(X)$ for topological spaces. We note that in both contexts, $c$ is a ring homomorphism, whereas $r$ is only a group 
 homomorphism. As is well-known $c\circ r=1+\bar{~}$, where in both contexts, $\bar{~}$ denotes the complex conjugation, and $r\circ c=\times 2$. 
 
 When $H\subset G$ is a closed subgroup of a compact connected Lie group and $X=G/H,$ the maps $c$ and $r$ commute with the $\alpha$-construction, i.e., 
 the following diagram commutes:

\begin{figure}[H]
            \[\begin{tikzcd}
         RO(H)\arrow[rr, ""] \arrow[d, "\alpha_0"]&&  RH \arrow[rr ,""]\arrow[d, "\alpha"]&& RO(H)\arrow[d, "\alpha_0"] \\
         KO^0(G/H) \arrow[rr, ""] && K(G/H)\arrow[rr ,""]&&  KO^0(G/H)
    \end{tikzcd}\]

    \centering
\caption{}
\end{figure}

 For a more detailed discussion, see \cite[\S11, Chapter 13]{Hu}.
We remark that if $[V]\in RO(H)$ is the restriction to $H$ of a 
$G$-representation, then $\alpha_0([V])=\dim V$ \cite[p.209, Example 2]{GHV}. An analogous statement holds for $\alpha.$

Since $G$ is compact, $RO(G)$ (resp. $RG$) is a free abelian group with basis the set of isomorphism classes of irreducible 
$G$-representation over $\mathbb R$ (resp. $\mathbb C$), and since $c\circ r=\times 2$, $c$ is a monomorphism. 

Also, $c(\Lambda_\mathbb R^j([V]))=\Lambda^j_{\mathbb C}(c[V])~\forall j\ge 0.$ for any real representation $V$.  Thus $c$ is a $\lambda$-ring homomorphism.   (However, $\mathbb R=\Lambda^2(r(\mathbb C))\ne r(\Lambda^2(\mathbb C))=r(0)=0$.)
Analogous statements hold for vector bundles.
These properties will be used below without explicit reference. 

\subsection{The tangent bundle of $\E_6/\Spin(10)$}\label{immersion}
Let us denote by $\lambda_{1, \mathbb R}\in RO(\Spin(10))$ the class of the standard representation of $\SO(10)$ on $\mathbb R^{10}$ regarded as a representation of $\Spin(10)$. Denote by $\lambda_{j,\mathbb R}$ the representation 
$\Lambda^j(\lambda_{1,\mathbb R})$.  
Thus $\lambda_{j,\mathbb R}\otimes \mathbb C=\lambda_j\in R(\Spin(10)).$
Note that $\lambda_{2,\mathbb R}=Lie(\Spin(10))$, the adjoint representation of $\Spin(10).$

\begin{proposition}\label{tangent1}
    Let $\tau$ be the tangent  bundle of $\E_6/\Spin(10)$. 
    Then 
    \[[ \tau]=[53-2\lambda_{1, \bb R}]\in KO(\E_6/\Spin(10)).\]
\end{proposition}
\begin{proof} 
For any homogeneous space $M=G/H$ one has the isomorphism 
$\tau$ is isomorphic to the vector bundle associated to the representation $\alpha_0(Lie(G)/Lie(H))$ of $H$. 
Therefore, we have (see \S\ref{candr}) 
\begin{equation}\label{tangent0}
[\tau]=\alpha_0[Lie(E_6)]-\alpha_0[Lie(\Spin(10)]
=78-\alpha_0[Lie(\Spin(10))]=78-\lambda_{2,\mathbb R}.
\end{equation}

    By Proposition \ref{restrictionformula} (iii),  $Lie(\E_6)\otimes \bb C=V_{\varpi_5}=\gamma\in R\E_6$ and restricts to $1+\lambda_2+\Delta_5^++\Delta_5^-$ as a representation of $\Spin(10)$. 
    Now $1+\lambda_2+\Delta^+_5+\Delta^-_5= c(1+\lambda_{2,\mathbb R}+r(\Delta^+_5))$ since $\Delta^-_5$ is the complex conjugate of $\Delta^+_5$.  So $c([Lie(\E_6)-Lie (\Spin (10))])=c(1+r(\Delta^+_5))$ in $R\Spin(10).$   Since $c:RO(\Spin(10))\to R\Spin(10)$ is a monomorphism, we have $[Lie(\E_6)-Lie(\Spin(10))]=1+r(\Delta^+_5)\in RO(\Spin(10)).$   Set $[V]:=r(\Delta^+_5)\in RO(\Spin(10))$.
    Then 
    \begin{equation}\label{tangent}
[\tau]=[1+V] \in KO(\E_6/\Spin(10)).
\end{equation}

Recall that $\rho(V_{\varpi_1})=1+\lambda_1+\Delta^+_5$ 
in $R\Spin(10)$ by Proposition \ref{restrictionformula}(i), and so $r(V_{\varpi_1})=2+2\lambda_{1,\mathbb R}+V$.   
Since $r(\alpha^\p)$ is restriction of the representation $r(\alpha)$ of $\E_6$ to $\Spin(10)$, 
it follows that the class of the real vector bundle associated 
to $2+2\lambda_1+V$ in $KO(\E_6/\Spin(10))$ is trivial of rank $54$, that is,  
$[2+2\lambda_{1,\mathbb R}+V]=[54]\in KO(\E_6/\Spin(10))$.
Substituting for $[V]$ in the expression for 
$[\tau]$ in Equation (\ref{tangent}), we obtain that 
$[\tau]+2[\lambda_{1,\mathbb R}]=53$. 
\end{proof}

We are ready to prove Theorem \ref{main2}.
\begin{proof} 
By Proposition \ref{tangent1}, we have 
$[\tau]+[2\lambda_{1,\mathbb R}]=53.$
Since $\E_6/\Spin(10)$ has dimension $33$, we obtain an 
isomorphism of {\em vector bundles}
$\tau\oplus \nu\cong 53\epsilon$
where $\nu$ is associated to $2\lambda_{1,\mathbb R}.$  (Cf. \cite[Chapter 9, Theorem 1.5]{Hu}.)
It follows 
from Hirsch's immersion theorem \cite{hirsch}, \cite[Chapter 18, Remark 2.7]{Hu},  
that $\E_6/\Spin(10)$ can be immersed in $\mathbb R^{53}.$

Singhof and Wemmer (\cite[Theorem 2]{SW}) have shown that $M=\E_6/\Spin(10)$ is not stably parallelizable. In 
fact, they show ({\em ibid}, \S3) that the rational Pontjagin class $p_2(M)$ does not vanish.  Also Singhof  (\cite[ Cor. 3.3(a), p.105]{S}) has shown that $p_1(M)=0$  It follows that 
$p_2(\nu)=-p_2(\tau M)\ne 0$.   This yields a lower bound for the immersion dimension of $M$, namely $M$ that cannot be immersed in $\mathbb R^{40}.$   
\end{proof}

\bibliographystyle{abbrv}
\bibliography{references}

\begin{thebibliography}{10}

\bibitem{A6}
J.~F. Adams.
\newblock {\em Lectures on {L}ie groups.}
\newblock W. A. Benjamin, Inc., New York-Amsterdam,, 1969.

\bibitem{A3}
J.~F. Adams.
\newblock The fundamental representations of {$E_8$}.
\newblock In {\em Conference on algebraic topology in honor of {P}eter {H}ilton
  ({S}aint {J}ohn's, {N}fld., 1983)}, volume~37 of {\em Contemp. Math.}, pages
  1--10. Amer. Math. Soc., Providence, RI, 1985.

\bibitem{AdamsMahmudMimura}
J.~F. Adams.
\newblock {\em Lectures on exceptional {L}ie groups}.
\newblock Chicago Lectures in Mathematics. University of Chicago Press,
  Chicago, IL, 1996.

\bibitem{AGUZ}
E.~Antoniano, S.~Gitler, J.~Ucci, and P.~Zvengrowski.
\newblock On the {$K$}-theory and parallelizability of projective {S}tiefel
  manifolds.
\newblock {\em Bol. Soc. Mat. Mexicana (2)}, 31(1):29--46, 1986.

\bibitem{AH1}
M.~F. Atiyah and F.~Hirzebruch.
\newblock Riemann-{R}och theorems for differentiable manifolds.
\newblock {\em Bull. Amer. Math. Soc.}, 65:276--281, 1959.

\bibitem{BF}
N.~E. Barufatti and D.~Hacon.
\newblock {$K$}-theory of projective {S}tiefel manifolds.
\newblock {\em Trans. Amer. Math. Soc.}, 352(7):3189--3209, 2000.

\bibitem{CE}
H.~Cartan and S.~Eilenberg.
\newblock {\em Homological algebra}.
\newblock Princeton University Press, Princeton, N. J., 1956.

\bibitem{FH}
W.~Fulton and J.~Harris.
\newblock {\em Representation theory: a first course}, volume 129.
\newblock Springer Science \& Business Media, 2013.

\bibitem{GHV}
W.~Greub, S.~Halperin, and R.~Vanstone.
\newblock {\em Connections, curvature, and cohomology. {V}ol. {II}: {L}ie
  groups, principal bundles, and characteristic classes}.
\newblock Pure and Applied Mathematics, Vol. 47-II. Academic Press [Harcourt
  Brace Jovanovich, Publishers], New York-London, 1973.

\bibitem{gulliot}
P.~Guillot.
\newblock The representation ring of a simply connected {L}ie group as a
  {$\lambda$}-ring.
\newblock {\em Comm. Algebra}, 35(3):875--883, 2007.

\bibitem{Hall}
B.~C. Hall.
\newblock {\em Lie groups, {L}ie algebras, and representations}, volume 222 of
  {\em Graduate Texts in Mathematics}.
\newblock Springer-Verlag, New York, 2003.

\bibitem{hirsch}
M.~W. Hirsch.
\newblock Immersions of manifolds.
\newblock {\em Trans. Amer. Math. Soc.}, 93:242--276, 1959.

\bibitem{HO}
L.~Hodgkin.
\newblock The equivariant {K}\"{u}nneth theorem in {$K$}-theory.
\newblock In {\em Topics in {$K$}-theory. {T}wo independent contributions},
  Lecture Notes in Math., Vol. 496, pages 1--101. Springer, Berlin, 1975.

\bibitem{Hu}
D.~Husemoller.
\newblock {\em Fibre bundles}, volume~20 of {\em Graduate Texts in
  Mathematics}.
\newblock Springer-Verlag, New York, third edition, 1994.

\bibitem{P}
H.~V. Pittie.
\newblock Homogeneous vector bundles on homogeneous spaces.
\newblock {\em Topology}, 11:199--203, 1972.

\bibitem{podder}
S.~Podder.
\newblock {\em {K}-theory of real Grassmann manifolds and of {$\mathrm
  E_6/\mathrm Spin(10)$}}.
\newblock Thesis (Ph.D.)--I.I.T. Madras, Chennai. 2023.

\bibitem{podder-sam=nkaran}
S.~Podder and P.~Sankaran.
\newblock {K}-theory of real {G}rassmann manifolds.
\newblock {\em Homology, Homotopy and Applications}, 25(1):401--419, 2023.

\bibitem{RO}
A.~Roux.
\newblock Application de la suite spectrale d'{H}odgkin au calcul de la
  {$K$}-th\'{e}orie des vari\'{e}t\'{e}s de {S}tiefel.
\newblock {\em Bull. Soc. Math. France}, 99:345--368, 1971.

\bibitem{S}
W.~Singhof.
\newblock Parallelizability of homogeneous spaces. {I}.
\newblock {\em Math. Ann.}, 260(1):101--116, 1982.

\bibitem{SW}
W.~Singhof and D.~Wemmer.
\newblock Parallelizability of homogeneous spaces. {II}.
\newblock {\em Math. Ann.}, 274(1):157--176, 1986.

\bibitem{wang}
H.-C. Wang.
\newblock Closed manifolds with homogeneous complex structure.
\newblock {\em American Journal of Mathematics}, 76(1):1--32, 1954.

\end{thebibliography}

\end{document}